\documentclass{article}

\usepackage[top=30truemm,bottom=30truemm,left=25truemm,right=25truemm]{geometry}

\usepackage{amssymb, amsmath, amsthm}

\newtheorem{theorem}{Theorem}
\newtheorem{proposition}{Proposition}
\newtheorem{lemma}{Lemma}
\newtheorem{corollary}{Corollary}

\def\pd#1{\partial_{#1}}

\def\ann#1{\mathrm{Ann}(#1)}

\def\ini#1{\mathrm{in}_{(0,1)}(#1)}
\def\etr#1{\mathrm{etr}\left(#1\right)}

\title{The Annihilating Ideal of the Fisher Integral}
\author{Tamio Koyama
}
\date{\today}

\begin{document}
\maketitle

\begin{abstract}
In this paper, we discuss a system of differential equations 
for the Fisher integral on the special orthogonal group.
Especially, we explicitly give a set of linear differential operators
which generates the annihilating ideal of the Fisher integral,
and we prove that the annihilating ideal is a maximal left ideal 
of the ring of differential operators with polynomial coefficients.
Our proof is given by a discussion concerned with an annihilating ideal
of a Schwartz distribution associated with the Haar measure on 
the special orthogonal group.
We also give differential operators annihilating the Fisher integral for 
the diagonal matrix by a new approach.

\noindent{\bf Keywords}\/:
Weyl algebra,
Haar measure,
Annihilating ideal,
Special orthogonal groups.
\end{abstract}

\section{Introduction}
We denote by $SO(n)$ the special orthogonal group with size $n$, i.e., 
$$
SO(n)=\left\{y\in \mathbf R^{n\times n} \mid y^\top y=e, \det y = 1\right\}.
$$
Here, $y^\top$ is the transpose of $y$ and $e$ is the identity matrix.
The Haar measure $\mu$ on $SO(n)$ is a probability measure on $SO(n)$
which satisfies the equation
$$
\int_{SO(n)} f(a^\top y) \mu(dy) = \int_{SO(n)} f(y) \mu(dy) 
$$
for arbitrary continuous function $f$ on $SO(n)$ and any $a\in SO(n)$.
The Fisher integral on the special orthogonal group is 
an integral with $n\times n$ matrix parameter $x$ given by 
$$
\int_{SO_(n)} \exp\left(\mathrm{tr}(xy)\right)\mu(dy)
=\int_{SO_(n)} \exp\left(\sum_{i,j=1}^nx_{ij}y_{ij}\right)\mu(dy).
$$

The Fisher integral is the normalizing constant of the Fisher distribution
which is discussed in \cite{SSTOT2013}.
Numerical calculations of the normalizing constant is very important 
in the point of view of applications in statistics and 
the holonomic gradient method (HGM) has been applied to such problems.
For example, see \cite{NNNOSTT2011}, \cite{HNTT2013}, 
\cite{Sei-Kume2013}, and \cite{KNNT2012-1}.
In order to apply the HGM, we need theoretical consideration of 
differential equations for each problem.
In the case of the Fisher integral, \cite{SSTOT2013} gives 
a system of differential equations for the Fisher integral and 
they conjectured that the system induces a holonomic ideal.
We prove their conjecture positively in this paper.
Furthermore, we also prove that this holonomic ideal is an maximal left ideal
of a Weyl algebra and consequently it is the annihilating ideal of 
the Fisher integral.

In order to show these results, we consider 
a Schwartz distribution which is associated with the Haar measure on 
the special orthogonal group.
We obtain a generating set of the annihilating ideal for the Fisher integral
from the annihilating ideal of the Schwartz distribution.

In the application to statistics, differential equations for 
the Fisher integral in the diagonal case is more important.
Actuary, we need a holonomic system for the function
$$
f(x_1,\dots,x_n)
=\int_{SO_(n)} \exp\left(\sum_{i=1}^nx_iy_{ii}\right)\mu(dy).
$$
In \cite{SSTOT2013}, a system of differential equations for $f(x)$ was 
obtained from a differential equation for the matrix hypergeometric function
$_0F_1$.
It is also proved that differential operators
$$
(x_i^2-x_j^2)\pd{x_i}\pd{x_j}-(x_j\pd{x_i}-x_i\pd{x_j})-(x_i^2-x_j^2)\pd{x_k}
\quad \left((i,j,k)=(1,2,3),(2,3,1),(3,1,2)\right)
$$
annihilates $f(x)$ in the case where $n=3$.
We give these system of differential equations by a new approach
in our paper.

The construction of this paper is as follows.
In Section \ref{24}, we review the basic notions of Weyl algebra
and give some lemmas, which we need in the later section.
In Section \ref{10}, we give a definition of a distribution associated with 
the Haar measure on $SO(n)$.  And we give an generating set of the annihilating 
ideal of the distribution.
In Section \ref{25}, we give an generating set of the annihilating ideal of
the Fisher integral. We also give a system differential equations for the 
Fisher integral in the diagonal case by a new approach.

\section{Weyl Algebra}\label{24}
In this section, we review basic notions in the theory of 
algebraic analysis, and give a lemma concerning 
characteristic varieties and maximal ideals 
which we need for a calculation in the later section.

In the first, we review some notions in algebraic geometry.
Let $n$ be a natural number.
We denote by $x=(x_1,\dots,x_n)$ the standard coordinate of 
the affine space $X:=\mathbf C^n$.
Let $\mathbf C[x]:= \mathbf C[x_1,\dots,x_n]$ be the polynomial ring
with variables $x_1, \dots x_n$.
A subset of the space $X$ is called {\it an algebraic set}
if it can be written as 
$$
\mathbf V(f_1,\dots,f_k)
:= \{a=(a_1,\dots,a_n)\in\mathbf X:f_1(a)=\cdots=f_k(a)=0\}
$$
by finite polynomials $f_1,\dots,f_k\in\mathbf C[x]$.
Any ideal $I$ of the polynomial ring $\mathbf C[x]$ defines
an algebraic set:
\begin{equation}\label{6}
\mathbf V(I):=\{a=(a_1,\dots,a_n)\in\mathbf X: f(a)=0,\,f\in I\}
\end{equation}
On the other hand, for any algebraic set $V\subset\mathbf C^n$,
we can obtain an ideal of $\mathbf C[x]$ by
\begin{equation}\label{7}
\mathbf I(V):=\{f\in\mathbf C[x]: f(x)=0,\,x\in V\}.
\end{equation}
An algebraic set $V$ is said to be irreducible
if it satisfies the following property:
if there exist two algebraic set, $V_1$ and $V_2$,
such that $V=V_1\cup V_2$, then we have $V=V_1$ or $V=V_2$.

For an algebraic set $V$, {\it the Krull dimension} of $V$ is 
the supuremum of the length $k$ of strictly increasing sequence of 
irreducible algebraic sets such that
$$
V_1\subsetneq \cdots \subsetneq V_k \subset V
$$
For any ideal $I\subset\mathbf C[x]$, the Krull dimension of $\mathbf V(I)$
is equals to the degree of the Hilbert polynomial of $I$
(for example, see e.g. \cite{Mumford1999}).
Algorithms computing Krull dimensions for given algebraic set are given
in \cite{CLO1992}.
In the section \ref{10} and \ref{25},
we utilize the methods for computing Krull dimension in this book.

In the next, we review the basic notions of the Weyl algebra.
Let us consider the ring of partial differential operators
with polynomial coefficients
$D_X:=\mathbf C\langle x_i,\pd{i}:i=1,\dots,n\rangle$.
Here, we put $\pd{i}:=\partial/\partial x_i\,(i=1,\dots,n)$.
It is also called the Weyl algebra
The Weyl algebra $D_X$  naturally acts on
the space  of the smooth functions on $X$,
the space of the Schwartz distributions on $X$,
and so on.
When $f$ is a smooth function or a Schwartz distribution on $X$,
we denote by $p\bullet f$ the function which is given by 
applying a differential operator $p\in D_X$ on $f$.
We call a left ideal
\[
\left\{p\in D_X: p\bullet f=0\right\}
\]
in $D_X$ the annihilating ideal of $f$,
and denote it by $\ann{f}$.

Any element $p$ of $D_X$ can be written uniquely as 
the form of a finite sum
$
p=\sum c_{\alpha\beta}x^\alpha\pd{}^\beta\,(c_{\alpha\beta}\in\mathbf C)
$.
Here, $\alpha, \beta\in\mathbf Z_{\geq 0}^n$ are multi indices, 
and $x^\alpha\pd{}^\beta=\prod_{i=1}^nx_i^{\alpha_i}\pd{i}^{\beta_i}$.
For multi index $\beta\in\mathbf Z_{\geq 0}^n$, 
we put $|\beta|=\sum_{i=1}^n\beta_i$.
For a differential operator $p=\sum c_{\alpha\beta}x^\alpha\pd{}^\beta\in D_X$,
we put an element $\mathrm{in}_{(0,1)}(p)$ of 
a polynomial ring $\mathbf C[x,\xi]:= \mathbf C[x_i,\xi_i:i=1,\dots,n]$
as 
$$
\mathrm{in}_{(0,1)}(p)=\sum_{|\beta|=m}c_{\alpha\beta}x^\alpha\xi^\beta
\quad \left(m = \max\{|\beta|:c_{\alpha\beta}\neq 0\}\right).
$$
For a left ideal $I$ of $D_X$, we call an ideal of $\mathbf C[x,\xi]$
defined by 
$$
\left\{\mathrm{in}_{(0,1)}(p) : p \in I\right\}
$$
{\it the characteristic ideal} of $I$, and denote it by $\mathrm{in}_{(0,1)}(I)$.

{\it The characteristic variety } of $I$ is 
the algebraic set $\mathbf V(\mathrm{in}_{(0,1)}(I))\subset X\times\Xi$
defined by the characteristic ideal $\mathrm{in}_{(0,1)}(I)$.
Here, $\Xi:=\mathbf C^n$ and we denote by $\xi=(\xi_1,\dots,\xi_n)$
the standard coordinate system of the space $\Xi$.
When $I\neq D_X$, 
the Krull dimension of the characteristic variety of $I$ 
is not less than $n$
(The Bernstein inequality \cite{Bjork1979},\cite{Coutinho1995}).
When the Krull dimension of the characteristic variety equals to $n$,
the left ideal $I$ is said to be holonomic.

For calculations in the later sections,
we prepare the following lemma
concerning with characteristic varieties and maximal ideals:
\begin{lemma}\label{5}
If left ideals $I$ and $J$ of $D_X$ satisfies $I\subsetneq J$,
then we have $\mathrm{in}_{(0,1)}(I) \subsetneq \mathrm{in}_{(0,1)}(J)$.
\end{lemma}
\begin{proof}
In this proof, we utilize the theory of Gr\"obner basis
(for example, see e.g., \cite{CLO1992}).
Let polynomials $p_1,\dots, p_k$ form a Gr\"obner basis of $I$
with respect to the order $<:=<_{(0,1)}$.
Since the left ideal $J$ is strictly larger than $I$, we can take $p\in J-I$.
Since $p$ is not included in $I$, 
we obtain the remainder $r\neq 0$ after division of $p$ by $p_1,\dots,p_k$.
By replacing $p$ to the remainder $r$,
we can assume $\mathrm{in}_<(p) \notin \mathrm{in}_<(I)$
without loss of generality.
On the other hand, we have $\mathrm{in}_<(p)\in\mathrm{in}_<(J)$ by $p\in J$.
Therefore, $\mathrm{in}_<(J)$ is strictly larger than $\mathrm{in}_<(I)$.
Suppose $\mathrm{in}_{(0,1)}(I) = \mathrm{in}_{(0,1)}(J)$.
For arbitrary $f \in \mathrm{in}_<(J)$, 
there exists $p\in J$ such that $\mathrm{in}_<(p)=f$.
Since we have 
$\mathrm{in}_{(0,1)}(p)\in\mathrm{in}_{(0,1)}(J) = \mathrm{in}_{(0,1)}(I)$,
there exists $q\in I$ such that $\mathrm{in}_{(0,1)}(q)=\mathrm{in}_{(0,1)}(p)$.
Here, we have
\begin{align*}
\mathrm{in}_<(q) 
&= \mathrm{in}_<(\mathrm{in}_{(0,1)}(q))\\
&= \mathrm{in}_<(\mathrm{in}_{(0,1)}(p))\\
&= \mathrm{in}_<(p) =f.
\end{align*}
This contradict that 
$\mathrm{in}_<(J)$ is strictly larger than $\mathrm{in}_<(I)$.
Therefore, we have $\mathrm{in}_{(0,1)}(I) \subsetneq \mathrm{in}_{(0,1)}(J)$.
\end{proof}
\begin{lemma}\label{9}
Suppose a left ideal $I \subset D_X$ is holonomic.
Then $I$ is a maximal left ideal of $D_X$ 
if $\mathrm{in}_{(0,1)}(I)$ is a prime ideal.
\end{lemma}
\begin{proof}
This proof will be by contradiction.
Suppose the left ideal $I \subset D_X$ is not maximal, 
then we have  $I=D_X$ or 
there exists a left ideal $J$ such that $I \subsetneq J \subsetneq D_X$.
In the case of $I=D_X$, the characteristic variety is the empty set.
This contradict that $I$ is a holonomic ideal.
Hence, the left ideal $J$ exists.
By lemma \ref{5}, 
$\mathrm{in}_{(0,1)}(J)$ is strictly larger than $\mathrm{in}_{(0,1)}(I)$.
Since $\mathrm{in}_{(0,1)}(I)$ is a prime ideal, 
$\sqrt{\mathrm{in}_{(0,1)}(I)}$ equals to $\mathrm{in}_{(0,1)}(I)$.
And this implies 
$$
\sqrt{\mathrm{in}_{(0,1)}(I)}
=\mathrm{in}_{(0,1)}(I)
\subsetneq \mathrm{in}_{(0,1)}(J)
\subset \sqrt{\mathrm{in}_{(0,1)}(J)}.
$$
Hence, we have
$\sqrt{\mathrm{in}_{(0,1)}(I)}\subsetneq\sqrt{\mathrm{in}_{(0,1)}(J)}$.
By the Hilbert's Strong Nullstellensatz, 
we have 
$
\mathbf I(\mathbf V(\mathrm{in}_{(0,1)}(I)))
  \subsetneq
\mathbf I(\mathbf V(\mathrm{in}_{(0,1)}(J)))
$.
Here, we use the notations of \eqref{6} and \eqref{7}.
By \cite[Chapter\,1, Section\,4, Proposition\,8]{CLO1992},
$
\mathbf V(\mathrm{in}_{(0,1)}(I))
  \supsetneq
\mathbf V(\mathrm{in}_{(0,1)}(J))
$
holds.
Let $V$ and $W$ be the characteristic varieties $D_X/I$ and $D_X/J$ 
respectively.
By the above arguments, we have $V\supsetneq W$. 
Since $\ini I$ is a prime ideal, $V$ is an irreducible algebraic set.
By the Bernstein inequality, the Krull dimension of $W$ equals to $n$.
Hence, there exist irreducible algebraic sets $W_i\,(i=1,\dots,n)$
such that $W_1\subsetneq \dots, \subsetneq W_n \subset W$.
Adding $V$ to the sequence, 
we obtain a strictly increasing sequence 
$W_1\subsetneq \dots, \subsetneq W_n \subsetneq V$
of irreducible algebraic sets with length $n+1$.
However, this contradict that the dimension of $V$ equals to $n$.
\end{proof}

The Fourier transformation $\mathcal F$
(resp. the inverse Fourier transformation formation $\mathcal F^{-1}$)
 for differential operators is 
a morphism of $\mathbf C$-algebra form $D_X$ to $D_X$ defined by
\begin{align*}
\mathcal F(x_i)&=-\pd{i},&\mathcal F(\pd{i})&=x_i,\\
\mathcal F^{-1}(x_i)&=\pd{i},&\mathcal F^{-1}(\pd{i})&=-x_i.
\end{align*}
Since the Fourier transformation is an isomorphism of $\mathbf C$-algebra,
we have the following lemma:
\begin{lemma}\label{38}
If a left ideal $I\subset D_X$ is maximal,
then 
$$
\mathcal F(I)=\{\mathcal F(p) \mid p\in I\}
$$
and 
$$
\mathcal F^{-1}(I)=\{\mathcal F^{-1}(p) \mid p\in I\}
$$
are also maximal left ideals of $D_X$.
\end{lemma}

\section{Haar Measure}\label{10}
%\section{Distribution associated with Haar Measure}\label{10}
In this section, we review the Haar measure on the special orthogonal groups,
and define a Schwartz distribution associated with this Haar measure.
Let $n$ be a natural number, $Y$ be a set consisting of
$n\times n$ matrices whose components are real numbers.  
For $1\leq i, j\leq n$, 
let $y_{ij}$ be a function from $Y$ to $\mathbf R$.
For each point $y$ in $Y$, $y_{ij}$ corresponds to $(i,j)$-component of $y$.
The functions $y_{ij}$ give a local coordinate system of $Y$.

The following relations define a sub-manifold of $Y$: 
\begin{align*}
%\delta_{ij}-\sum_{k=1}^n y_{ik}y_{jk} \quad (1\leq i,j\leq n)
y^\top y &= e\\
\det y &= 1,
\end{align*}
Here, $y^\top$ denotes the transpose of $y$ and $e$ denotes the identity matrix.
By the product of matrices, this sub-manifold defines a Lie group.
This Lie group called the special orthogonal group, and denoted by $SO_n$.

On the special orthogonal groups, there uniquely exists 
the measure $\mu$ which satisfies the following properties:
\begin{align*}
\int_{SO_n} f(y)\mu(dy) 
&= \int_{SO_n} f(z^\top y)\mu(dy) \quad f\in C^\infty(SO_n),\,z\in SO_n\\
\int_{SO_n} \mu(dy) 
&= 1
\end{align*}
Here, we denote by $C^\infty(SO_n)$ the set of continuous functions on $SO_n$.
We call the measure $\mu$ the Haar measure on the special orthogonal group.

Let us define a Schwartz distribution on the space $Y$
associated with the Haar measure $\mu$ on $SO_n$.
We denote by $C^\infty_0(Y)$ the set of continuous functions on $Y$
with compact supports.
For a function $f$ on $Y$, $f\restriction_{SO_n}$ denotes the restriction of 
$f$ to $SO_n$.
The map from the functional space $C^\infty_0(Y)$ to $\mathbf R$ 
defined by
\[
\varphi \mapsto \int f\restriction_{SO_n}(y) \mu(dy)
\quad \left(\varphi\in C^\infty_0(Y) \right)
\]
gives a Schwartz distribution on $Y$.
We denote this distribution by the same notation $\mu$.

We denote by
$D_Y:=\mathbf C\langle y_{ij},\pd{ij}: 1\leq i,j\leq n\rangle$
the ring of differential operators with polynomial coefficient
with variable $y_{ij}\,(1\leq i,j\leq n)$.
Here, we put $\pd{ij}:=\partial/\partial y_{ij}$.
The annihilating ideal $\ann{\mu}$ of the distribution $\mu$ on $Y$ 
is a left ideal of $D_Y$.
In this section, we explicitly give a set of differential operators which 
generates the annihilating ideal $\ann{\mu}$.
The first step for this purpose is giving 
some differential operators which annihilate the distribution $\mu$.
The second step is studying the properties 
the ideal $I$ generated by these differential operators.
By these properties, we prove $I=\ann{\mu}$.

\begin{lemma}\label{1}
The following differential operators annihilate $\mu$.
\begin{align}
\label{2}
& \sum_{k=1}^n (y_{ki}\pd{kj}-y_{kj}\pd{ki}) \quad (1\leq i < j \leq n)\\
\label{3}
& \delta_{ij}-\sum_{k=1}^n y_{ki}y_{kj},\,
  \delta_{ij}-\sum_{k=1}^n y_{ik}y_{jk}
  \quad (1\leq i\leq j \leq n)\\
\label{4}
& 1-\det y
\end{align}
Here, $\delta_{ij}$ is the Kronecker's symbol.
\end{lemma}
\begin{proof}
Let $\varphi$ be a smooth function on $Y$ with compact support.
Since the functions \eqref{3} vanish on $SO_n$, we have 
\begin{align*}
\left\langle 
  \left(\delta_{ij}-\sum_{k=1}^n y_{ki}y_{kj}\right)\mu, 
  \varphi 
\right\rangle
&= 
\left\langle 
  \mu,  
  \left(\delta_{ij}-\sum_{k=1}^n y_{ki}y_{kj}\right)\varphi
\right\rangle \\
&=
\int_{SO_n} \left(\delta_{ij}-\sum_{k=1}^n y_{ki}y_{kj}\right)\varphi(y) \mu(dy)
=0.
\end{align*}
Hence, the differential operator \eqref{3} annihilates $\mu$.
Analogously, we can prove \eqref{4} annihilates $\mu$.

Let $E_{ij}\ (1\leq i<j\leq n)$ be a $n\times n$ matrix 
whose $(k,\ell)$ element is $\delta_{ik}\delta_{j\ell}-\delta_{jk}\delta_{i\ell}$, 
and $c(t)=\exp(tE_{ij})$ for $t\in \mathbf R$.
For a smooth function $f(y)$ on $Y$, we denote $R_{c(t)}f(y)=f(y\cdot c(t))$.
Let $v_{ij}$ be a vector field on $Y$ defined as 
\[
(v_{ij})_yf = \frac{\partial R_{c(t)}f}{\partial t} \restriction_{t=0}(y)
\quad (y \in Y,\, f \in C^\infty(Y)).
\]
It is easy to show that 
\[
v_{ij} = \sum_{k=1}^n (y_{ki}\pd{kj}-y_{kj}\pd{ki}).
\]
Note that the differential operator $\pd{ij}=\partial/\partial y_{ij}$
can be regarded as a vector field on $Y$.
Since the measure $\mu$ is right invariant under $SO_n$, we have
\begin{align*}
\left\langle \sum_{k=1}^n (y_{ik}\pd{jk}-y_{jk}\pd{ik})\mu, \varphi \right\rangle
&= 
-\left\langle \mu, \sum_{k=1}^n (y_{ik}\pd{jk}-y_{jk}\pd{ik})\varphi \right\rangle\\
&= 
-\int_{SO_n}  (v_{ij}\varphi)(y)\mu(dy)\\
&=
-\int_{SO_n}
\frac{\partial R_{c(t)}\varphi}{\partial t} \restriction_{t=0}(y)\mu(dy)\\
%&=
%-\int_{SO_n}
% \lim_{t\rightarrow 0}
% \frac{\varphi(c(-t)\cdot Y)-\varphi(Y)}{t} \mu(dY)\\
&=
-\lim_{t\rightarrow 0}\int_{SO_n}
\frac{\varphi(y\cdot c(t))-\varphi(y)}{t} \mu(dy)\\
&=
0.
\end{align*}
Hence, the differential operator \eqref{2} annihilates $\mu$.
\end{proof}
Let $I$ be an ideal generated by the differential operators 
\eqref{2},\eqref{3}, and \eqref{4}.
By lemma \ref{1}, we have $I\subset\ann{\mu}$.
For the opposite inclusion,
it is enough to prove the following proposition:
\begin{proposition}\label{8}
The left ideal $I$ is a holonomic ideal, 
and the characteristic ideal of $I$ is a prime ideal.
\end{proposition}
In fact, by this proposition and lemma \ref{9}, 
the left ideal $I$ is a maximal ideal of $D_Y$.
Since $\ann{\mu}\neq D_Y$, we have $I=\ann{\mu}$.

Let $J$ be an ideal of the polynomial ring 
$\mathbf C[y,\xi]:=\mathbf C[y_{ij},\xi_{ij}:1\leq i,j\leq n]$
generated by \eqref{3},\eqref{4} and 
$$
\sum_{k=1}^n (y_{ki}\xi_{kj}-y_{kj}\xi_{ki}) \quad (1\leq i < j \leq n).
$$
Obviously, $J\subset\ini{I}$ holds
and we have $\mathbf V(J)\supset\mathbf V(\ini{I})$.
Now, let us suppose $J$ is a prime ideal 
and the Krull dimension of $\mathbf V(J)$ equals to $n\times n$.
By the Bernstein's inequality, the Krull dimension of $\mathbf V(\ini{I})$
is not less than $n\times n$.
Then, we have $\mathbf V(J) = \mathbf V(\ini{I})$.
By the Strong Nullstellensatz, we have $\sqrt{J} = \sqrt{\ini{I}}$.
Here, utilizing the assumption that $J$ is prime, we have $J = \sqrt{J}$.
Consequently, we have $J = \sqrt{\ini{I}} \supset \ini{I}$.
Hence, $J=\ini{I}$ holds.
This shows that the Krull dimension of $\mathbf V(\ini{I})$ equals to $n\times n$,
i.e., the left ideal $I$ is holonomic and the ideal $\ini{I}$ is prime.

In order to prove proposition \ref{8},
it is enough to show the following two statement:
$J$ is a prime ideal 
and the Krull dimension of $\mathbf V(J)$ equals to $n\times n$.
For this purpose, we define an ideal $J'$ such that
$\mathbf C[y,\xi]/J\cong \mathbf C[y,\xi]/J'$,
and show that $J'$ is prime and the Krull dimension of $\mathbf V(J')$ 
equals to $n\times n$.

Let $J'$ be an ideal of $\mathbf C[y,\xi]$ generated by 
\eqref{3},\eqref{4}, and 
\begin{equation}\label{11}
 \xi_{ij} - \xi_{ji} \quad  (1\leq i< j \leq n).
\end{equation}

\begin{lemma}
The quotient ring $\mathbf C[y,\xi] / J$ is isomorphic to $\mathbf C[y,\xi] / J'$ 
as $\mathbf C$-algebra. 
\end{lemma}
\begin{proof}
Define $\mathbf C$-algebra homomorphisms 
$\phi :\mathbf C[y,\xi] \rightarrow \mathbf C[y,\xi]$ and 
$\psi :\mathbf C[y,\xi] \rightarrow \mathbf C[y,\xi]$ 
as 
\begin{align*}
\phi(y_{ij}) &= y_{ij}, \quad 
\phi(\xi_{ij}) = \sum_{k=1}^n y_{ik}\xi_{kj}  \quad (1\leq i, j \leq n),\\
\psi(y_{ij}) &= y_{ij}, \quad
\psi(\xi_{ij}) = \sum_{k=1}^n y_{ki}\xi_{kj} \quad (1\leq i, j \leq n).
\end{align*} 
By some calculations, we can prove the following formula:
\begin{align}
\phi\left(\sum_{k=1}^n(y_{ki}\xi_{kj}-y_{kj}\xi_{ki})\right)
&=\xi_{ij}-\xi_{ji}
-\sum_{\ell=1}^n\left(\delta_{i\ell}-\sum_{k=1}^n y_{ki}y_{k\ell}\right)\xi_{\ell j} 
\nonumber\\
&\quad +\sum_{\ell=1}^n\left(\delta_{j\ell}-\sum_{k=1}^n y_{kj}y_{k\ell}\right)\xi_{\ell i}
\label{12}\\
\psi\left(\xi_{ij}-\xi_{ji}\right)
&=\sum_{k=1}^n (y_{ki}\xi_{kj}-y_{kj}\xi_{ki})
\label{13}\\
\phi\psi(\xi_{ij})
&=\xi_{ij}
  -\sum_{\ell=1}^n\xi_{\ell j}\left(\delta_{i\ell}-\sum_{k=1}^ny_{ki}y_{k\ell}\right)
\label{14}\\
\psi\phi(\xi_{ij})
&=\xi_{ij}
  -\sum_{\ell=1}^n\xi_{j\ell}\left(\delta_{i\ell}-\sum_{k=1}^ny_{ik}y_{\ell k}\right)
\label{15}
\end{align}
Let
$p_1:\mathbf C[y,\xi]\rightarrow \mathbf C[y,\xi]/J \, (p_1(f) = \overline f)$
and 
$p_2:\mathbf C[y,\xi]\rightarrow\mathbf C[y,\xi]/J' \, (p_2(f) = \overline f)$
be the projections, then we have
$\ker(p_2\phi)=J$ and $\ker(p_1\psi)=J'$.
In fact, 
$\ker(p_2\phi)\supset J$ follows by \eqref{12}, 
and $\ker(p_1\psi)\supset J'$ follows by \eqref{13}.
Let $f\in\ker(p_2\phi)$, then we have $\phi(f)\in J'$.
By \eqref{13}, we have $\psi\phi(f)\in J$.
Since we also have $f-\psi\phi(f)\in J$ by \eqref{15},
$f$ is an element of $J$.
Hence, $\ker(p_2\phi)=J$ holds.
Analogously, if we take $f\in\ker(p_1\psi)$, 
then $\psi(f)$ is an element of $J$.
The equation \eqref{12} implies $\psi\phi(f)\in J'$.
The equation \eqref{14} also implies $f-\phi\psi(f)\in J'$,
and we have  $f\in J'$.
Hence, $\ker(p_1\psi)=J'$ holds also.

By the isomorphism theorem,
we have two morphisms, 
$\mathbf C[y,\xi]/J\rightarrow \mathbf C[y,\xi]/J'$ and
$\mathbf C[y,\xi]/J'\rightarrow \mathbf C[y,\xi]/J$.
We can show by some calculations 
that their compositions equal to the identity morphisms.
\end{proof}
In order to prove that the ideal $J'$ is prime, 
we utilize a tensor product of $\mathbf C$-algebras.
The following lemma is well known:
\begin{lemma}\label{21}
Let $\mathbf C[x]:=\mathbf C[x_1,\dots,x_n]$, 
$\mathbf C[y]:=\mathbf C[y_1,\dots,y_m]$, 
and $\mathbf C[x,y]:=\mathbf C[x_1,\dots,x_n,y_1,\dots,y_m]$
be polynomial rings.
And we denote by
$\iota_1:\mathbf C[x]\rightarrow\mathbf C[x,y]$ 
and
$\iota_2:\mathbf C[y]\rightarrow\mathbf C[x,y]$
the immersion maps.
Let $I_1$ and $I_2$ be ideals of $\mathbf C[x]$ and $\mathbf C[y]$
respectively.
Then, there exists the following isomorphism:
$$
\mathbf C[x]/I_1 \otimes_{\mathbf C} \mathbf C[y]/I_2
\cong \mathbf C[x,y]/I
\quad\left(
      \overline{f}\otimes\overline{g}\mapsto\overline{\iota_1(f)\iota_2(g)}
    \right)
$$
where $I=\mathbf C[x,y]\iota_1(I_1)+\mathbf C[x,y]\iota_2(I_2)$.
\end{lemma}
\begin{proof}
see, e.g., \cite[I,\S 6,Proposition\ 1.]{Mumford1999}.
\end{proof}
Now, let us compute the ideal $J'$.
\begin{lemma}
The ideal $J'$ is a prime ideal 
and the Krull dimension of $\mathbf V(J')$ equals to $n\times n$. 
\end{lemma}
\begin{proof}
In the first, we calculate the dimension of $\mathbf V(J')$.
Let $J'_1$ be an ideal of $\mathbf C[y_{ij}:1\leq i,j\leq n]$
generated by the polynomials \eqref{3} and \eqref{4}.
The algebraic set $\mathbf V(J'_1)$ equals to the special orthogonal group,
and it's Krull dimension is $n(n-1)/2$.
Moreover, 
the ideal $J'_1$ is prime by \cite[p147.Theorem(5.4c)]{Weyl1946}.
Especially, the algebraic set $\mathbf V(J'_1)$ is irreducible.

Let $J'_2$ be an ideal of $\mathbf C[\xi_{ij}:1\leq i,j\leq n]$
generated by the polynomials \eqref{11}.
Let $<$ be a graded lexicographic order 
which satisfies $\xi_{ij}>\xi_{ji}\,(1\leq i<j\leq n)$.
The polynomials \eqref{11} form a Gr\"obner basis of $J'_2$ 
with respect to the order $<$.
Hence, the Krull dimension of $\mathbf V(J'_1)$ equals to $n(n+1)/2$.
Besides,
the quotient ring $\mathbf C[\xi_{ij}:1\leq i,j\leq n]/J'_2$ is 
isomorphic to a polynomial ring $\mathbf C[\xi_{ij}: 1\leq i \leq j\leq n]$.
In fact, let
$
\varphi: 
\mathbf C[\xi_{ij}:1\leq i,j\leq n]/J'_2 
\rightarrow 
\mathbf C[\xi_{ij}: 1\leq i \leq j\leq n]
$
and 
$
\psi: 
\mathbf C[\xi_{ij}: 1\leq i \leq j\leq n]
\rightarrow 
\mathbf C[\xi_{ij}:1\leq i,j\leq n]/J'_2 
$
be morphisms defined by 
\begin{align*}
\varphi(\overline \xi_{ij})
&=
\begin{cases}
\xi_{ij} &(i\leq j)\\
\xi_{ji} &(i > j)
\end{cases}, &
\psi(\xi_{ij}) &= \overline \xi_{ij},
\end{align*}
then $\varphi\psi$ and $\psi\varphi$ are isomorphisms.
Since $\mathbf C[\xi_{ij}: 1\leq i \leq j\leq n]$ is an integral domain, 
$\mathbf C[\xi_{ij}:1\leq i,j\leq n]/J'_2$ is also an integral domain.
Hence, $J'_2$ is a prime ideal.

By $\mathbf V(J')=\mathbf V(J'_1)\times\mathbf V(J'_2)$,
the Krull dimension of $\mathbf V(J')$ equals $n(n-1)/2+n(n+1)/2=n^2$.

In the second, we show that the ideal $J'$ is prime.
Since $\mathbf V(J'_1)$ and $\mathbf V(J'_2)$ are irreducible algebraic sets,
their product $\mathbf V(J')=\mathbf V(J'_1)\times\mathbf V(J'_2)$ is 
irreducible also.
Hence, the coordinate ring $\mathbf C[y,\xi]/\sqrt{J'}$ of $\mathbf V(J')$
is an integral domain.
Also, we have an isomorphism between the coordinate rings:
$$
\mathbf C[y,\xi]/\sqrt{J'}
\rightarrow 
\mathbf C[y]/J'_1 \otimes_{\mathbf C}\mathbf C[\xi]/J'_2
\quad
\left(\overline{f(y)g(\xi)} \mapsto \overline{f(y)}\otimes\overline{g(\xi)}\right).
$$
By lemma \ref{21}, we have an isomorphism
$$
\mathbf C[y]/J'_1 \otimes_{\mathbf C}\mathbf C[\xi]/J'_2
\rightarrow \mathbf C[x,\xi]/J'
\quad \left(
         \overline{f(y)}\otimes\overline{g(\xi)}\rightarrow\overline{f(y)g(\xi)}
     \right).
$$
These isomorphisms give an isomorphism
$$
\mathbf C[y,\xi]/\sqrt{J'}\rightarrow\mathbf C[y,\xi]/J'
\quad \left(\overline{f(y)g(\xi)}\rightarrow\overline{f(y)g(\xi)}\right).
$$
Since this isomorphism implies 
that $\mathbf C[y,\xi]/J'$ is an integral domain,
$J'$ is a prime ideal.
\end{proof}

Therefore, we have the following theorem.
\begin{theorem}\label{34}
The differential operators \eqref{2},\eqref{3}, and \eqref{4} generate
the annihilating ideal of the Schwartz distribution $\mu$ 
associated with the Haar measure on the special orthogonal group.
\end{theorem}

\section{The Fisher Integral}\label{25}
In this section, we use the notations in Section \ref{10}.
Especially, $\mu$ denotes the Haar measure on the special orthogonal group.
For a square matrix $a$, we put 
$\etr{a}:=\exp\left(\mathrm{tr}\left(a\right)\right)$.

The following lemma is useful \cite{OST2003}:
\begin{lemma}\label{17}
Let 
$ D_n:= \mathbf C\langle x_1, \dots, x_n, \partial _1, \dots, \partial _n \rangle$
be the ring of differential operators with polynomial coefficients.
Let $u$ be a Schwartz distribution and $f$ be a polynomial in  
$\mathbf C[x_1,\dots,x_n]$. We put $f_i :=\partial  f/\partial  x_i$.
Suppose a left ideal $I$ of $D_n$ is holonomic and annihilates $u$.
Then, the left ideal $J$ generated by 
$$
\left\{
P(x_1, \dots, x_n;\partial _{x_1}-f_1, \dots, \partial _{x_n}-f_n)
\vert
P(x_1, \dots, x_n;\partial _{x_1}, \dots, \partial _{x_n})\in I
\right\}
$$
is holonomic and annihilates the distribution $e^fu$.
\end{lemma}

{\it The Fisher integral} is the following function defined by 
an integration on the special orthogonal group:
$$
f(x) = \int_{SO(n)} \etr{xy}\mu(dy)
\quad \left(x\in\mathbf C^{n\times n} \right).
$$
Here, $x=(x_{ij})$ is an $n\times n$ matrix over the field of real numbers.
In this section, 
we explicitly give the annihilating ideal of the Fisher integral $f(x)$
as an application of Theorem \ref{34}.

In \cite{SSTOT2013}, it is proved that 
the Fisher integral $f(x)$ annihilated by the following differential operators:
\begin{align}
\label{35}
& \sum_{k=1}^n (x_{ki}\pd{x_{kj}}-x_{kj}\pd{x_{ki}}) \quad (1\leq i < j \leq n)\\
\label{36}
& \delta_{ij}-\sum_{k=1}^n \pd{x_{ki}}\pd{x_{kj}},\,
  \delta_{ij}-\sum_{k=1}^n \pd{x_{ik}}\pd{x_{jk}}
  \quad (1\leq i\leq j \leq n)\\
\label{37}
& 1-\det \pd{x},
\end{align}
Here, we denote by $\det \pd{x}$ the determinant 
of the matrix whose $(i.j)$-element is $\pd{x_{ij}}$.

\begin{theorem}
The annihilating ideal of the Fisher integral $f(x)$ is generated 
by the differential operators \eqref{35}, \eqref{36}, and \eqref{37}.
\end{theorem}
\begin{proof}
Let $I$ be a left ideal generated by the differential operators
\eqref{35}, \eqref{36}, and \eqref{37}.
Since the ideal $I$ equals to $\mathcal F^{-1}(\ann{\mu})$ 
and $\ann{\mu}$ is maximal, $I$ is also maximal by Lemma \ref{38}.
Since $\ann{f}\neq D_X$ and $I$ is maximum,
we have $I=\ann{f}$.
\end{proof}
\begin{corollary}
The left ideal generated by \eqref{35}, \eqref{36}, and \eqref{37}
is a maximal ideal of $D_X$. 
And this ideal is a holonomic ideal of $D_X$.
\end{corollary}

In application to statistics, the case where the matrix $x$ is diagonal is 
more important. In this case, the Fisher integral is a function 
with respect to $x_1,\dots, x_n$, which are the diagonal elements of $x$, 
defined by
\begin{equation}\label{40}
\tilde f(x_1,\dots,x_n)
=
\int_{SO(n)}\exp\left(\sum_{i=1}^nx_iy_{ii}\right)\mu(dy).
\end{equation}
A system of differential equations for \eqref{40} was given 
in \cite{SSTOT2013}.
\begin{proposition}[\cite{SSTOT2013}]\label{41}%% (24)
The differential operator
\begin{equation}\label{43}
\pd{x_i}^2
-\sum_{k\neq i}\frac{1}{x_i^2-x_k^2}\left(x_i\pd{x_i}-x_k\pd{x_k}\right)
-1
\quad (i=1,\dots,n)
\end{equation}
annihilates \eqref{40}.
\end{proposition}
When $n=3$, there are extra differential operators annihilating \eqref{40}.
\begin{proposition}[\cite{SSTOT2013}]\label{42}
When $n=3$, 
the differential operators
$$
(x_i^2-x_j^2)\pd{x_i}\pd{x_j}-(x_j\pd{x_i}-x_i\pd{x_j})-(x_i^2-x_j^2)\pd{x_k}
\quad \left((i,j,k)=(1,2,3),(2,3,1),(3,1,2)\right)
$$
annihilates \eqref{40}.
\end{proposition}
As an application of Theorem \ref{34}, we give new proofs for 
Proposition \ref{41} and Proposition \ref{42}.

\begin{proof}[Proof of Proposition \ref{41}]
By Lemma \ref{17}, the integrand $\exp\left(\sum_{i=1}^nx_iy_{ii}\right)\mu(dy)$
is annihilated by 
\begin{align*}
p_{ij}
&:= \sum_{k=1}^n (y_{ki}\pd{y_{kj}}-y_{kj}\pd{y_{ki}})-y_{ji}x_j+y_{ij}x_i 
\quad (1\leq i < j \leq n),\\
\tilde p_{ij}
&:= \sum_{k=1}^n (y_{ik}\pd{y_{jk}}-y_{jk}\pd{y_{ik}})-y_{ij}x_j+y_{ji}x_i 
\quad (1\leq i < j \leq n),\\
%& \delta_{ij}-\sum_{k=1}^n y_{ki}y_{kj},\,
%  \delta_{ij}-\sum_{k=1}^n y_{ik}y_{jk}
%  \quad (1\leq i\leq j \leq n),\\
%& 1-\det y,\\
q_i&:= \pd{x_i}-y_{ii}\quad(1\leq i \leq n).
\end{align*}
Here, we regard $\exp\left(\sum_{i=1}^nx_iy_{ii}\right)\mu(dy)$ as 
a Schwartz distribution.
Considering the elements of 
$$
\frac{1}{x_i^2-x_j^2}
\begin{pmatrix} x_i & x_j\\ x_j & x_i \end{pmatrix}
\begin{pmatrix} p_{ij} \\ \tilde p_{ij} \end{pmatrix}
$$
for $1\leq i<j\leq n$,
we have that the differential operators
\begin{align*}
&y_{ij}+\frac{1}{(x_i+x_j)(x_i-x_j)}\left(x_i
\sum_{k=1}^n (y_{ki}\pd{y_{kj}}-y_{kj}\pd{y_{ki}})
+x_j
\sum_{k=1}^n (y_{ik}\pd{y_{jk}}-y_{jk}\pd{y_{ik}})
\right),\\
&y_{ji}+\frac{1}{(x_i+x_j)(x_i-x_j)}\left(x_j
\sum_{k=1}^n (y_{ki}\pd{y_{kj}}-y_{kj}\pd{y_{ki}})
+x_i
\sum_{k=1}^n (y_{ik}\pd{y_{jk}}-y_{jk}\pd{y_{ik}})
\right)
\end{align*}
annihilate the integrand.
Since the differential operator 
$-1+\sum_{j=1}^dy_{ij}^2$ annihilates the integrand,
the differential operator
\begin{align*}
&-1+\pd{x_i}^2
-\sum_{j\neq i}\frac{y_{ij}}{x_i^2-x_j^2}
\left(x_i
\sum_{k=1}^n (y_{ki}\pd{y_{kj}}-y_{kj}\pd{y_{ki}})
+x_j
\sum_{k=1}^n (y_{ik}\pd{y_{jk}}-y_{jk}\pd{y_{ik}})
\right)\\
&=
-1+\pd{x_i}^2
-\sum_{j\neq i}\frac{1}{x_i^2-x_j^2}
\left(x_i
\sum_{k=1}^n (y_{ki}\pd{y_{kj}}-y_{kj}\pd{y_{ki}})
+x_j
\sum_{k=1}^n (y_{ik}\pd{y_{jk}}-y_{jk}\pd{y_{ik}})
\right)y_{ij}\\
&\quad 
+\sum_{j\neq i}\frac{1}{x_i^2-x_j^2}
\left(x_iy_{ii}-x_jy_{jj}\right)
\end{align*}
also annihilates the integrand.
Hence, we have that the operator \eqref{43} annihilates \eqref{40}.
\end{proof}

Our proof of Proposition \ref{42} utilizes
the following well-known formula (see, e.g., \cite[2.5.2 THEOREM]{prasolov1994problems})\/:
for $n\times n$ regular matrix $A$ and $I,J\subset [n]$ with $|I|=|J|$, we have 
\begin{equation}\label{44}
[A^{-1}]_{I,J} = \det A^{-1}[A]_{I',J'}.
\end{equation}
Here, $I', J'$ denote the subsets of indices complementary to $I, J$, 
and $[A]_{I,J}$ denotes the minor of $A$ corresponding to $I, J$.
\begin{proof}[Proof of Proposition \ref{42}]
We can assume that $(i,j,k)=(1,2,3)$ without loss of generalities.
Let $I=\{1, 2\}$ and $J=\{3\}$. Then the equation \eqref{44} implies that 
the all $3\times 3$ special orthogonal matrix $A=(a_{ij})$ satisfies the equation
$$
\det\begin{pmatrix} a_{11}&a_{21}\\a_{12}&a_{22}\end{pmatrix}
= \det\begin{pmatrix} a_{33}\end{pmatrix}.
$$
Hence, the operator
$$
y_{11}y_{22}-y_{12}y_{21}-y_{33}
$$
annihilates the integrand $\exp\left(\sum_{i=1}^nx_iy_{ii}\right)\mu(dy)$.
By the analogous way to the proof of Proposition \ref{41}, we have that 
the differential operator 
\begin{align*}
&\pd{x_1}\pd{x_2}
+\frac{y_{12}}{x_2^2-x_1^2}
\left(
  x_2\sum_{k=1}^3 (y_{k2}\pd{y_{k1}}-y_{k1}\pd{y_{k2}})
 +x_1\sum_{k=1}^3 (y_{2k}\pd{y_{1k}}-y_{1k}\pd{y_{2k}})
\right)
-\pd{x_3}\\
&=\pd{x_1}\pd{x_2}
+
\frac{1}{x_2^2-x_1^2}
\left(
  x_2\sum_{k=1}^3 (y_{k2}\pd{y_{k1}}-y_{k1}\pd{y_{k2}})
 +x_1\sum_{k=1}^3 (y_{2k}\pd{y_{1k}}-y_{1k}\pd{y_{2k}})
\right)y_{12}
-\pd{x_3}\\
&\quad+
\frac{1}{x_2^2-x_1^2}\left(y_{11}x_2-y_{22}x_1\right)
\end{align*}
annihilates the integrand.
Hence, the differential operator
$$
\pd{x_1}\pd{x_2}
+\frac{1}{x_2^2-x_1^2}\left(x_2\pd{x_1}-x_1\pd{x_2}\right)
-\pd{x_3}
$$
annihilates the integral \eqref{40}.
\end{proof}

\section*{Acknowledgement}
This work was supported by JSPS KAKENHI Grant Number JP 18J01507.

\iffalse%\iftrue%

\bibliographystyle{plain}
\bibliography{reference}

\else

\fi

\end{document}